\documentclass[a4paper,11pt]{article} 
\usepackage[english]{babel}
\usepackage{anysize}
\marginsize{3.5cm}{3.5cm}{2.5cm}{2.5cm}

\title{A Note on Large Deviations for 2D Coulomb Gas \\with Weakly Confining Potential}

\author{
\ Adrien Hardy 
\footnote{Institut de Math\'ematiques de Toulouse, Universit\'e de Toulouse, 31062 Toulouse, France.} \hspace{2 dd}\footnote{Department of Mathematics, Katholieke Universiteit Leuven, Celestijnenlaan 200 B,
3001 Leuven, Belgium. Email address: adrien.hardy@wis.kuleuven.be} }

\usepackage{fancyhdr}
\usepackage{amsfonts}
\usepackage{amsmath} 
\usepackage{color} 
\usepackage{amssymb} 
\usepackage{amsthm}
\usepackage[colorlinks=true, linkcolor=blue, citecolor=blue]{hyperref}
\usepackage[fixlanguage]{babelbib}

\numberwithin{equation}{section}

\newtheorem{theorem}{Theorem}[section]
\newtheorem{lemma}[theorem]{Lemma}
\newtheorem{corollary}[theorem]{Corollary}
\newtheorem{proposition}[theorem]{Proposition}

\theoremstyle{definition}

\newtheorem{Remark}[theorem]{Remark}
\newenvironment{remark}{\begin{Remark}\rm}{\end{Remark}}
\newenvironment{example}{\begin{Example}\rm}{\end{Example}}

\newtheorem{Example}[theorem]{Example}

\newcommand{\eq}{\begin{equation}}
\newcommand{\qe}{\end{equation}}

\newcommand{\R}{\mathbb{R}}
\newcommand{\C}{\mathbb{C}}
\newcommand{\M}{\mathcal{M}}
\newcommand{\F}{\mathcal{F}}
\newcommand{\V}{\boldsymbol{\mathcal{V}}}
\newcommand{\B}{\mathcal{B}}
\newcommand{\K}{\mathcal{K}}
\newcommand{\U}{\mathcal{U}}

\newcommand{\p}{\mathbb{P}}
\newcommand{\s}{\mathcal{S}}
\newcommand{\X}{\mathcal{X}}

\newcommand{\bs}{\boldsymbol}

\begin{document}
\maketitle 

\begin{abstract}
We investigate a  Coulomb gas in a potential satisfying a weaker growth assumption than usual and establish a large deviation principle for its empirical measure. As a consequence the empirical measure is seen to converge towards a non-random limiting measure, characterized by a variational principle from logarithmic potential theory, which may not have compact support. The proof of the large deviation upper bound is based on a compactification procedure which may be of help for further large deviation principles.
\end{abstract}


\section{Introduction and statement of the result}
Given an infinite  closed subset $\Delta$ of $\C$, consider the distribution of $N$ particles $x_1,\ldots,x_N$ living on $\Delta$ which interact like a Coulomb gas at  inverse temperature $\beta>0$ under an external potential. Namely, let $\p_N$ be the probability distribution on $\Delta^N$ with density 
\eq
\label{density}
\frac{1}{Z_N}\prod_{1\leq i<j\leq N}|x_i-x_j|^\beta\prod_{i=1}^Ne^{-NV(x_i)},
\qe
where the so-called potential $V:\Delta\rightarrow\R$ is a continuous function which, provided $\Delta$ is unbounded, grows   sufficiently fast as $|x|\rightarrow\infty$ so that 
\eq
\label{ZN}
Z_N=\int\cdots\int_{\Delta^N} \prod_{1\leq i<j\leq N}|x_i-x_j|^\beta\prod_{i=1}^Ne^{-NV(x_i)}dx_i<+\infty.
\qe
For $\Delta=\R$ and $\beta=1$ (resp. $\beta=2$ and $4$) such a density is known to match with the joint eigenvalue distribution of a $N\times N$ orthogonal (resp. unitary and unitary symplectic) invariant Hermitian random matrix \cite{M}.  A similar observation can be made when $\Delta=\C$ (resp. the unit circle $\mathbb T$,  the real half-line $\R_+$, the segment $[0,1]$) by considering normal matrix models \cite{CZ} (resp. the $\beta$-circular ensemble, the $\beta$-Laguerre ensemble, the $\beta$-Jacobi ensemble, see \cite{Fo} for an overview).

In this work, our interest lies in the limiting  global  distribution of the $x_i$'s as $N\rightarrow\infty$, that is the convergence of the empirical measure
\eq
\mu^N=\frac{1}{N}\sum_{i=1}^N\delta_{x_i}
\qe
in the case where $\Delta$ is unbounded and $V$ satisfies a weaker growth assumption than usually presented in the literature, see \eqref{weakgrowth}. Note the $\mu^N$'s  are random variables taking their values in the space $\M_1(\Delta)$ of probability measures on $\Delta$, that we equip with the usual weak topology. 

When $\Delta=\R$, the almost sure convergence of $(\mu^N)_N$ towards a non-random  limit $\mu^*_V$ is classically known to hold under the hypothesis that there exists  $\beta'>1$ satisfying $\beta'\geq \beta$ such that 
\eq
\label{growth}
\liminf_{|x|\rightarrow\infty}\frac{V(x)}{\beta'\log|x|}>1,
\qe
that is, as $|x|\rightarrow\infty$, the confinement effect due to the potential  $V$ is stronger than the repulsion between the $x_i$'s. The limiting distribution $\mu^*_V$ is then characterized as the unique minimizer of the functional
\eq
\label{weightedlogenergy}
I_V(\mu)=\iint F_V(x,y)d\mu(x)d\mu(y), \qquad \mu\in\M_1(\Delta),
\qe
where we introduced the following variation of the weighted logarithmic kernel
\eq
F_V(x,y)=\frac{\beta}{2}\log\frac{1}{|x-y|}+\frac{1}{2}V(x)+\frac{1}{2}V(y),Ê\qquad x,y\in\Delta.
\qe
A stronger statement, first established by Ben Arous and Guionnet for a Gaussian potential $V(x)=x^2/2$ \cite{BAG} and later extended to arbitrary continuous potential $V$ satisfying the growth condition \eqref{growth} \cite[Theorem 2.6.1]{AGZ} (see also \cite[Theorem 5.4.3]{HP} for a similar statement with a slightly stronger growth assumption on $V$),  is that $(\mu^N)_N$ satisfies a large deviation principle (LDP) on $\M_1(\Delta)$ in the scale $N^2$ and  good rate function $I_V-I_V(\mu^*_V)$. It is moreover known that $\mu^*_V$ has a compact support \cite[Lemma 2.6.2]{AGZ}. A similar result is known to hold when $\Delta=\C$, see e.g. \cite[Theorem 5.4.9]{HP}.

It is the aim of this work to show that such statements still hold, except that $\mu^*_V$ may not have compact support,  when one allows the confining effect of the potential $V$ to be of the same order of magnitude than the repulsion between the $x_i$'s. Namely, we consider the following weaker  growth condition: there exists  $\beta'>1$ satisfying $\beta'\geq \beta$ such that 
\eq
\label{weakgrowth}
\liminf_{|x|\rightarrow\infty}\Big\{V(x)-\beta'\log|x|\Big\}>-\infty.
\qe
We provide a statement when $\Delta=\R$ or $\C$, and discuss later the case of more general $\Delta$'s. More precisely, we will establish  the following.

\begin{theorem} 
\label{th2} Let $\Delta=\R$ or $\C$. Under the growth assumption \eqref{weakgrowth},  
\begin{itemize}
\item[{\rm (a)}] The level set $\big\{\mu\in\M_1(\Delta):\,I_V(\mu)\leq \alpha\big\}$ is compact for any $\alpha\in\R$.
\item[{\rm (b)}] $I_V$ admits a unique minimizer $\mu_V^*$ on $\M_1(\Delta)$.
\item[{\rm (c)}] For any closed set $\F\subset \M_1(\Delta)$,
\[
\limsup_{N\rightarrow\infty}\frac{1}{N^2}\log \p_N\Big(\mu^N\in\F\Big)\leq - \inf_{\mu\in\F}\Big\{I_V(\mu)-I_V(\mu^*_V)\Big\}.
\]
\item[{\rm (d)}] For any open set $\mathcal O\subset \M_1(\Delta)$,
\[
\liminf_{N\rightarrow\infty}\frac{1}{N^2}\log \p_N\Big(\mu^N\in\mathcal O\Big)\geq - \inf_{\mu\in\mathcal O}\Big\{I_V(\mu)-I_V(\mu^*_V)\Big\}.
\]
\end{itemize}
\end{theorem}
Note that \eqref{weakgrowth}, together with the inequality $|x-y|\leq (1+|x|)(1+|y|)$, $x,y\in\C$, yields \eqref{ZN} and that $F_V$ is bounded from below, so that $I_V$ is well defined on $\M_1(\Delta)$. 

A consequence of Theorem \ref{th2} (b) and (c), together with  the Borel-Cantelli Lemma, is the almost sure convergence of $(\mu^N)_N$ towards $\mu^*_V$ in the weak topology of $\M_1(\Delta)$. Namely, if  $\p$ stands for the probability measure induced by the  product probability space $\bigotimes_{N}\big(\Delta^N,\p_N\big)$, we have

\begin{corollary}
\label{convergence}
\[
\p\Big(\mu^N \mbox{converges weakly as } N\rightarrow\infty \mbox{ to $\mu^*_V$}\Big)= 1.
\]
\end{corollary}
Let us now discuss few  examples arising from random matrix theory where the limiting distribution $\mu^*_V$ has unbounded support.
\begin{example}\textbf{(Cauchy ensemble)}
\label{noncompact}
On the space $\mathcal{H}_N(\C)$ of $N\times N$ Hermitian complex matrices, consider the probability distribution
\[
\frac{1}{Z_N}\det(I_N+X^2)^{-N}dX,
\]
where $I_N\in\mathcal{H}_N(\C)$ is the identity matrix, $dX$ the Lebesgue measure of $\mathcal{H}_N(\C)\simeq\R^{N^2}$ and $Z_N$ a normalization constant. Such a matrix model is a variation of the Cauchy ensemble \cite[Section 2.5]{Fo}. Performing a spectral decomposition and integrating out the eigenvectors, it is known that the induced  distribution for the eigenvalues is  given by \eqref{density} with $\Delta=\R$, $\beta=2$, $V(x)=\log(1+x^2)$, and some new normalization constant $Z_N$. One can then compute, see Remark \ref{examplecontinued} below, that the minimizer of \eqref{weightedlogenergy}  is the Cauchy distribution
\eq
\label{cauchylaw}
d\mu^*_V(x)=\frac{1}{\pi(1+x^2)}dx,
\qe
where $dx$ is the Lebesgue measure on $\R$.
\end{example}

\begin{example}\textbf{(Spherical ensemble)}
\label{noncompact2}
Given $A$ and $B$ two independent $N\times N$ matrices  with i.i.d. standard complex Gaussian entries, it is known that  the $N$ zeros of the random polynomial $\det(A-zB)$ (i.e. the eigenvalues of $AB^{-1}$ when $B$ is invertible) are distributed according to  \eqref{density} with $\Delta=\C$, $\beta=2$, $V(x)=\log(1+|x|^2)$ (up to a negligible correction), see \cite[Section 3]{Kr}. One may also consider the probability distribution on the space $\mathcal{N}_N(\C)$ of $N\times N$ normal complex matrices given by
\[
\frac{1}{Z_N}\det(I_N+X^*X)^{-N}dX,
\]
where $I_N\in\mathcal{N}_N(\C)$ is the identity matrix, $dX$ the Riemannian volume form on $\mathcal{N}_N(\C)$ induced by the Lebesgue measure of the space of $N\times N$ complex matrices ($\simeq \C^{N^2}$),  $Z_N$ a normalization constant, and obtains the same Coulomb gas for the eigenvalue distribution \cite[Section 2]{CZ}. The minimizer of \eqref{weightedlogenergy}  is then the distribution 
\eq
\label{sphericallaw}
d\mu^*_V(x)=\frac{1}{\pi(1+|x|^2)^2}dx,
\qe
where $dx$ stands for the Lebesgue measure on $\C\simeq\R^2$, see Remark \ref{examplecontinued}.
\end{example}

\begin{remark}\textbf{(Exponential tightness and compactification)} \\
The proofs of the large deviation principles under the stronger growth assumption \eqref{growth}  presented in \cite{BAG}, \cite{HP}, \cite{AGZ} follow a classical strategy in large deviation principles theory (see e.g \cite{DZ} for an introduction), that is to  control   the deviations of $(\mu^N)_N$ towards arbitrary small balls of $\M_1(\Delta)$, and  then prove  an exponential tightness property for $(\mu^N)_N$ : there exists a sequence of compact sets $(\K_L)_L\subset \M_1(\Delta)$ such that 
\eq
\label{exptight}
\limsup_{L\rightarrow\infty}\limsup_{N\rightarrow\infty}\frac{1}{N^2}\log \p_N\Big(\mu^N\notin\K_L\Big)= -\infty.
\qe
The exponential tightness is actually used to establish the large deviation upper bound, and plays no role in the proof of the lower one. Under the weaker growth assumption \eqref{weakgrowth}, it is not clear to the author how to prove the exponential tightness for $(\mu^N)_N$ directly, and we thus prove Theorem \ref{th2} by using a different approach. We  adapt an idea of \cite{HK} and  map $\C$ onto the Riemann sphere $\s$, homeomorphic to the one-point compactification of $\C$ by the inverse stereographic projection $T$, then  push-forward $\M_1(\C)$ to $\M_1(\s)$, and take advantage that the latter set is compact for its weak topology. More precisely, it will be seen that it is enough to establish upper bounds for the deviations of $(T_*\mu^N)_N$, the push-forward of $(\mu^N)_N$ by $T$, towards arbitrary small balls of $\M_1(\s)$. The latter fact is possible thanks to the explicit change of metric induced by $T$.  
\end{remark}

Our approach is still available for a large class of supports $\Delta$ and for potentials $V$ satisfying  weaker regularity assumptions, justifying our choice to consider general $\Delta$'s.  Nevertheless, it is not the purpose of this note to establish in such a general setting the large deviation lower bound, which is a local property and in fact will be seen to be independent of the growth assumption for $V$. This is the reason why we restricted $\Delta$  to be $\R$ or $\C$ in Theorem \ref{th2}.

We first describe the announced compactification procedure in Section \ref{compactification}. Then, we study $(T_*\mu^N)_N$ and a related rate function in Section \ref{TstarmuNsection}.  From these informations, we are able to provide a proof for Theorem \ref{th2} in Section \ref{proofth2}. Finally, we discuss in Section \ref{generalizations}  some generalizations concerning the support of the Coulomb gas, the regularity of the potential and the compactification procedure of possible further interest.

\section{Proof of Theorem \ref{th2}}

We first describe the compactification procedure. In this subsection, $\Delta$ is an arbitrary unbounded closed subset of $\C$.

\subsection{Compactification}

\label{compactification}
We consider the Riemann sphere, here parametrized as the sphere of $\R^3$ centered in $(0,0,1/2)$ of radius $1/2$, 
\[
 \s = \Big\{ (x_1, x_2, x_3) \in \R^3 \mid x_1^2 + x_2^2 + (x_3- \tfrac{1}{2})^2 = \tfrac{1}{4} \Big\},
\]
and $T:\C\rightarrow \s$  the associated inverse stereographic projection, namely the map defined by 
\[
T(x)= \left(\frac{{\rm Re}(x)}{1+|x|^2},\frac{{\rm Im}(x)}{1+|x|^2},\frac{|x|^2}{1+|x|^2}\right),\qquad x\in\C.
\] 
It is known that $T$ an homeomorphism from $\C$ onto $\s\setminus\{\infty\}$, where $\infty=(0,0,1)$, so that $(\s,T)$ is a one-point compactification of $\C$. We write for convenience 
\eq
\label{deltaS}
\Delta_{\s}= {\rm clo}\big(T(\Delta)\big)=T(\Delta)\cup\{\infty\}
\qe
for the closure of $T(\Delta)$ in $\s$. For $\mu\in\M_1(\Delta)$, we  denote by $T_*\mu$  its push-forward by $T$, that is the measure on $\Delta_\s$ characterized by
\eq
\label{pushf}
\int_{\Delta_\s} f(z)dT_*\mu(z) = \int_\Delta f \big(T(x)\big)d\mu(x)
\qe
for every Borel function $f$ on $\Delta_\s$.   Then the following Lemma holds.

\begin{lemma} 
\label{homeo}
$T_*$ is an homeomorphism from $\M_1(\Delta)$ to 
\[
\big\{\mu\in\M_1(\Delta_\s):\,\mu(\{\infty\})=0\big\}.
\]
\end{lemma}

\begin{proof}
$T^*$ is clearly continuous. The inverse of $T_*$ is given by push backward via $T$, that is, for any $\mu\in\M_1(\Delta_\s)$ satisfying $\mu(\{\infty\})=0$, ${T_*}^{-1}\mu(A)=\mu(T(A))$ for all Borel set $A\subset\Delta_\s$.  To show the continuity of ${T_*}^{-1}$, consider a sequence $(\mu_N)_N$ in $\M_1(\Delta_\s)$ with weak limit $\mu$ and assume that $\mu_N(\{\infty\})=0$ for all $N$ and  $\mu(\{\infty\})=0$. Then, for any $\epsilon>0$, the outer regularity of $\mu$ and the weak convergence of $(\mu_N)_N$ towards $\mu$ yield the existence of a  neighborhood  $B\subset \Delta_\s$ of $\infty$ such that 
\[
\limsup_{N\rightarrow \infty}\mu_N(B)\leq \mu(B)\leq \epsilon,
\]
which equivalently means that  $({T_*}^{-1}\mu_N)_N$ is tight. As a consequence, since $f\circ {T}^{-1}$ is continuous on $\Delta_\s$ for any continuous function $f$ having compact support in $\Delta$, the continuity of ${T_*}^{-1}$ follows.

\end{proof}

The next step is to obtain an upper control on the deviation of $(T_*\mu^N)_N$ towards arbitrary small balls of $\M_1(\Delta_\s)$.

\subsection{Weak LDP upper bound for $(T_*\mu^N)_N$}
\label{TstarmuNsection}

In this subsection, $\Delta$ is an arbitrary unbounded closed subset of $\C$, the potential $V:\Delta\rightarrow\R\cup\{+\infty\}$ is a lower semi-continuous map satisfying the growth condition \eqref{weakgrowth}, and we assume there exists $\mu\in\M_1(\Delta)$ such that $I_V(\mu)<+\infty$.

The change of metric induced by $T$ is given by (see e.g. \cite[Lemma 3.4.2]{Ash})
\eq
\label{metric}
|T(x)-T(y)|=\frac{|x-y|}{\sqrt{1+|x|^2}\sqrt{1+|y|^2}}, \qquad x,y\in\C,
\qe
where $|\cdot|$  stands for the Euclidean norm of $\R^3$ (we identify $\C$ with $\{(x_1,x_2,x_3)\in\R^3 : \; x_3=0\}$). Note that  by letting $y\rightarrow+\infty$ in \eqref{metric}, squaring and using the Pythagorean theorem, one obtains the useful relation
\eq
\label{metric2}
1-|T(x)|^2=\frac{1}{1+|x|^2},\qquad x\in\C.
\qe
From the potential $V$ we then construct a potential $\V : \Delta_\s\rightarrow\R\cup\{+\infty\}$ in the following way. Set
\eq
\label{curlyV}
\V \big(T(x)\big)=V(x)-\frac{\beta}{2}\log(1+|x|^2), \qquad x\in\Delta,
\qe
and  
\eq
\label{curlyV2}
\V(\infty)=\liminf_{|x|\rightarrow\infty,\,x\in\Delta}\Big\{V(x)-\frac{\beta}{2}\log(1+|x|^2)\Big\}.
\qe
Note that  the growth assumption \eqref{weakgrowth}  is equivalent to $\V(\infty)>-\infty$, so that $\V$ is lower semi-continuous on $\Delta_\s$. As a consequence the kernel 
\eq
\label{Fbs}
F_{\V}(z,w)=\frac{\beta}{2}\log\frac{1}{|z-w|}+\frac{1}{2}\V(z)+\frac{1}{2}\V(w), \qquad z,w\in\Delta_\s,
\qe  
 is lower semi-continuous  and bounded from below on $\Delta_\s\times\Delta_\s$, and the functional
\eq
\label{RFs}
I_{\V}(\mu) = \iint F_{\V}(z,w)d\mu(z)d\mu(w), \qquad \mu\in\M_1(\Delta_\s),
\qe
is well-defined. One understands from \eqref{metric}, \eqref{curlyV} and  \eqref{pushf} that the potential $\V$ has been built so that the following relation  holds
\eq
\label{relation}
I_V(\mu)=I_{\V}\big(T_*\mu\big), \qquad \mu\in\M_1(\Delta).
\qe
Let us come back to Examples \ref{noncompact} and \ref{noncompact2}.

\begin{remark}\textbf{(Examples \ref{noncompact}, \ref{noncompact2}, continued)}
\label{examplecontinued}
For $\Delta=\R$ or $\C$, $\beta=2$ and $V(x)=\log(1+|x|^2)$, we have $\bs\V=0$ and thus from \eqref{relation}
\eq
I_V(\mu)=\iint\log\frac{1}{|z-w|}dT_*\mu(z)dT_*\mu(w),\qquad \mu\in\M_1(\Delta).
\qe
Note that if $\Delta=\R$ (resp. $\Delta=\C$) then $\Delta_\s=\s\cap\{(x_1,x_2,x_3)\in\R^3: \; x_2=0\}$ is a circle (resp.  $\Delta_\C=\s$ the full sphere).  By rotational invariance, the minimizer of 
\[
\iint \log\frac{1}{|z-w|}d\nu(z)d\nu(w),\qquad \nu\in \M_1( \Delta_\s)
\]
has to be the uniform measure ${\U}_{\Delta_\s}$ of  $\Delta_\s$, and  thus the minimizer $\mu_V^*$ of $I_V$ is given by the push-backward ${T_*}^{-1}\U_{\Delta_\s}$.  Thus, if $\Delta=\R$ (resp. $\Delta=\C$), an easy Jacobian computation involving polar (resp. spherical) coordinates yields that $\mu^*_V$ equals \eqref{cauchylaw} (resp. \eqref{sphericallaw}).
\end{remark}

Given a metric $d$ on $\M_1(\Delta_\s)$, compatible with its weak topology (such as the L\'evy-Prohorov metric, see \cite{D}), we denote for the associated balls
\[
\B(\mu,\delta)=\Big\{\nu\in\M_1(\Delta_\s) : \; d(\mu,\nu)<\delta\Big\}, \qquad \mu\in\M_1(\Delta_\s), \qquad \delta>0.
\]
The following Proposition gathers all the informations concerning $I_{\bs\V}$ and $(T_*\mu^N)_N$  needed to establish Theorem \ref{th2} in the next Section.

\begin{proposition} \
\label{TstarmuN}
\begin{enumerate}
\item[{\rm (a)}]
The level set $\big\{\mu\in\M_1(\Delta_\s):\,I_{\bs\V}(\mu)\leq \alpha\big\}$ is closed, and thus compact, for any $\alpha\in\R$.
\item[{\rm (b)}]
$I_{\bs\V}$ is strictly convex on the set where it is finite.
\item[{\rm (c)}]
For any $\mu\in\M_1(\Delta_\s)$, we have
\[
\limsup_{\delta\rightarrow0}\limsup_{N\rightarrow\infty}\frac{1}{N^2}\log\Big\{ Z_N\p_N\Big(T_*\mu^N\in\B(\mu,\delta)\Big)\Big\}\leq - I_{\bs\V}(\mu).
\]
\end{enumerate}
\end{proposition}

The proof of Proposition \ref{TstarmuN} is  somehow classical and  inspired from the ideas developed in \cite{BAG} (c.f.  also \cite{HP}, \cite{AGZ}, \cite{HK}). 
\begin{proof} 
(a) It is equivalent to show that $I_{\bs\V}$ is  lower semi-continuous. Since $F_{\V}$ is lower semi-continuous, there exists an increasing sequence $(F_{\V}^M)_M$ of continuous functions on $\Delta_\s\times\Delta_\s$ satisfying $F_{\V}=\sup_MF_{\V}^M$.  We obtain for any $\mu\in\M_1(\Delta_\s)$ by  monotone convergence
\[
I_{\V}(\mu)=\sup_M \iint F_{\V}^M(z,w)d\mu(z)d\mu(w),
\] 
and $I_{\V}$ is thus lower semi-continuous on $\M_1(\Delta_\s)$ being the supremum of a family of continuous functions.

(b) Denote for a (possibly signed) measure $\mu$ on $\s$ its logarithmic energy by
\eq
\label{logenergy}
I(\mu)=\iint \log\frac{1}{|x-y|}d\mu(x)d\mu(y)
\qe
when this integral makes sense, and note that if $\mu\in\M_1(\Delta_\s)$ then $I(\mu)\geq 0$.  Since  $\V$ is bounded from below and $\mu\mapsto\int\V(z)d\mu(z)$  is linear, it is enough to show that $\mu\mapsto I(\mu)$ is strictly convex on the set where it is finite. Given $\mu,\nu\in\M_1(\Delta_\s)$ having finite logarithmic energies, we have for any $0<t<1$
\[
I\big(t \mu +(1-t)\nu\big)= t I(\mu) + (1-t)I(\nu) - t(1-t)I(\mu-\nu).
\]
Moreover, since $I(\mu-\nu)\geq 0$ with equality if and only if $\mu=\nu$ \cite[Theorem 2.5]{CKL}, the strict convexity of $I$ where it is finite follows.

(c) Introduce for $i=1,\ldots,N$ the random variables $z_i=T(x_i)$ where the $x_i$'s are distributed according to \eqref{density} so that 
\eq
\label{empizi}
T_*\mu^N = \frac{1}{N}\sum_{i=1}^N \delta_{z_i}.
\qe
We can easily compute the distribution for the $z_i$'s induced  by \eqref{density}. Indeed, with $\V$ defined in \eqref{curlyV}--\eqref{curlyV2}, we obtain from the metric relations \eqref{metric}--\eqref{metric2} that
\begin{align*}
& \;\frac{1}{Z_N}\prod_{1\leq i<j\leq N}|x_i-x_j|^\beta\prod_{i=1}^Ne^{-NV(x_i)}dx_i\\
= & \;\frac{1}{Z_N}\prod_{1\leq i<j\leq N}|T(x_i)-T(x_j)|^\beta\prod_{i=1}^N\big(1-|T(x_i)|^2\big)^{\beta/2}e^{-N\big(V(x_i)-\frac{\beta}{2}\log(1+|x_i|^2)\big)}dx_i \\
= &\; \frac{1}{Z_N}\prod_{1\leq i<j\leq N}|z_i-z_j|^\beta\prod_{i=1}^N(1-|z_i|^2)^{\beta/2} e^{-N\bs\V(z_i)}d\lambda(z_i),
\end{align*}
where $\lambda$ stands for the push-forward by $T$ of (the restriction of) the Lebesgue measure on $\Delta$. As a consequence, we have

\begin{align}
\label{A1}
& Z_N\p_N\Big(T_*\mu^N\in\B(\mu,\delta)\Big) \nonumber \\
= &  \int\ldots\int_{\big\{\bs z\in \Delta_\s^N:\,T_*\mu^N\in\B(\mu,\delta)\big\}}\prod_{1\leq i<j\leq N}|z_i-z_j|^\beta\prod_{i=1}^N(1-|z_i|^2)^{\beta/2}e^{-N\bs\V(z_i)}d\lambda(z_i).
\end{align}
Then, with $F_{\V}$ defined in \eqref{Fbs}, one can write 
\begin{align}
\label{A2}
& \prod_{1\leq i<j\leq N}|z_i-z_j|^\beta\prod_{i=1}^N(1-|z_i|^2)^{\beta/2}e^{-N\bs\V(z_i)}d\lambda(z_i)\nonumber \\
=& \exp\Big\{-\sum_{1\leq i \neq j \leq N}F_{\bs\V}(z_i,z_j)\Big\}\prod_{i=1}^N(1-|z_i|^2)^{\beta/2}e^{-\bs\V(z_i)}d\lambda(z_i)\nonumber\\
= & \exp\Big\{-N^2\iint_{z\neq w}F_{\bs\V}(z,w)dT_*\mu^N(z)dT_*\mu^N(w)\Big\}\prod_{i=1}^N(1-|z_i|^2)^{\beta/2}e^{-\bs\V(z_i)}d\lambda(z_i).
\end{align}
With  $F^M_{\bs\V}$ as in the proof of Proposition \ref{TstarmuN}  (a) above, we have
\eq
\label{A3}
\iint_{z\neq w}F_{\bs\V}(z,w)dT_*\mu^N(z)dT_*\mu^N(w) \geq \iint_{z\neq w}F^M_{\bs\V}(z,w)dT_*\mu^N(z)dT_*\mu^N(w).
\qe
Moreover, since $\p_N$-almost surely
\[
T_*\mu^N\otimes T_*\mu^N\big(\{(x,y)\in\Delta_\s\times\Delta_\s:\,x=y\}\big)=\frac{1}{N},
\] 
we obtain on the event $\{T_*\mu^N\in\B(\mu,\delta)\}$  that
\begin{align}
\label{A4}
 &  \iint_{z\neq w}F^M_{\bs\V}(z,w)dT_*\mu^N(z)dT_*\mu^N(w) \nonumber \\
 \geq & \; \iint F^M_{\bs\V}(z,w)dT_*\mu^N(z)dT_*\mu^N(w)-\frac{1}{N}\max_{\Delta_\s\times\Delta_\s}F^M_{\bs\V}\nonumber \\
\geq  & \; \inf_{\nu\in\B(\mu,\delta)}\iint F^M_{\bs\V}(z,w)d\nu(z)d\nu(w)-\frac{1}{N}\max_{\Delta_\s\times\Delta_\s}F^M_{\bs\V}.
\end{align}
From \eqref{A1}--\eqref{A4} we find
\begin{align}
\label{A5}
& \log \Big\{Z_N\p_N\Big(T_*\mu^N\in\B(\mu,\delta)\Big)\Big\}\nonumber\\
\leq & \;-N^2\inf_{\nu\in\B(\mu,\delta)}\iint F^M_{\bs\V}(z,w)d\nu(z)d\nu(w) \\
 & \quad+N \left(\max_{\Delta_\s\times\Delta_\s}F^M_{\bs\V}+ \log \int_{\Delta_\s}(1-|z|^2)^{\beta/2}e^{-\bs\V(z)}d\lambda(z)\right). \nonumber
\end{align}
Note that by performing the change of variables $z=T(x)$, using \eqref{metric2} and the growth assumption \eqref{weakgrowth}, it follows that
\[
\int_{\Delta_\s}(1-|z|^2)^{\beta/2}e^{-\bs\V(z)}d\lambda(z) = \int_{\Delta}e^{-V(x)}dx<+\infty,
\]
and thus  \eqref{A5} yields
\eq
\label{A6}
\limsup_{N\rightarrow\infty}\frac{1}{N^2}\log \Big\{Z_N\p_N\Big(T_*\mu^N\in\B(\mu,\delta)\Big)\Big\}\leq -\inf_{\nu\in\B(\mu,\delta)}\iint F^M_{\bs\V}(z,w)d\nu(z)d\nu(w).
\qe
The continuity of the map
\[
\nu\mapsto \iint F^M_{\bs\V}(z,w)d\nu(z)d\nu(w)
\]
provides by  letting $\delta\rightarrow 0$ in \eqref{A6}
\eq
\label{A7}
\limsup_{\delta\rightarrow 0}\limsup_{N\rightarrow\infty}\frac{1}{N^2}\log \Big\{Z_N\p_N\Big(T_*\mu^N\in\B(\mu,\delta)\Big)\Big\}\leq -\iint F^M_{\bs\V}(z,w)d\mu(z)d\mu(w),
\qe
and (c) is finally deduced by monotone convergence letting $M\rightarrow\infty$ in \eqref{A7}.
\end{proof}

Equipped with Proposition \ref{TstarmuN}, we are now in position to prove Theorem \ref{th2} thanks to the compactification procedure described in Section \ref{compactification}.

\subsection{Proof of Theorem \ref{th2}}
\label{proofth2}
In this subsection, $\Delta=\R$ or $\C$, and $V:\Delta\rightarrow\R$ is a continuous map satisfying the growth assumption \eqref{weakgrowth}.

\begin{proof}[Proof of Theorem \ref{th2}] (a)
Since $I_{\bs\V}(\mu)=+\infty$ for all $\mu\in\M_1(\Delta_\s)$ such that $\mu(\{\infty\})>0$, we obtain from Lemma \ref{homeo} and \eqref{relation} that the levels sets of $I_V$ and $I_{\bs\V}$ are homeomorphic, namely for any $\alpha\in\R$
\[
T_*\Big\{\mu\in\M_1(\Delta):\,I_V(\mu)\leq \alpha\Big\} = \Big\{\mu\in\M_1(\Delta_\s):\,I_{\bs\V}(\mu)\leq \alpha\Big\}.
\]
Thus,  Theorem \ref{th2} (a) follows from Proposition \ref{TstarmuN} (a).

(b) Theorem \ref{th2} (a) yields the existence of minimizers for $I_V$ on $\M_1(\Delta)$. Since $T_*$ is a linear injection, it follows from \eqref{relation} and Proposition \ref{TstarmuN} (b) that $I_V$ is strictly convex on the set where it is finite, which warrants the uniqueness of the minimizer.

(c),(d) 
It is enough to show that for any closed set $\F\subset\M_1(\Delta)$,
\eq
\label{UB}
\limsup_{N\rightarrow\infty}\frac{1}{N^2}\log \Big\{Z_N \p_N\Big(\mu^N\in\F\Big)\Big\}\leq - \inf_{\mu\in\F}I_V(\mu),
\qe
and for any open set $\mathcal{O}\subset \M_1(\Delta)$,
\eq
\label{LB}
\liminf_{N\rightarrow\infty}\frac{1}{N^2}\log\Big\{ Z_N \p_N\Big(\mu^N\in\mathcal{O}\Big)\Big\}\geq - \inf_{\mu\in\mathcal{O}}I_V(\mu).
\qe
Indeed, by taking $\F=\mathcal{O}=\M_1(\Delta)$ in \eqref{UB} and \eqref{LB}, one  obtains
\[
\lim_{N\rightarrow\infty}\frac{1}{N^2}\log Z_N= -\inf_{\mu\in\M_1(\Delta)}I_V(\mu)=-I_V(\mu_V^*),
\]
the latter quantity being finite. 

Let us first show \eqref{UB}. We have for any closed set $\F\subset\M_1(\Delta)$ that
\eq
\label{C1}
 \p_N\Big(\mu^N\in\F\Big)  \leq  \p_N\Big(T_*\mu^N\in{\rm clo}(T_*\F)\Big),
\qe
where ${\rm clo}(T_*\F)$ stands for the closure of $T_*\F$ in $\M_1(\Delta_\s)$. Inspired from the proof of \cite[Theorem 4.1.11]{DZ}, we fix $\epsilon>0$, and introduce
\[
I^\epsilon_{\bs\V}(\mu)=\min\big(I_{\bs\V}(\mu)-\epsilon,1/\epsilon\big),\qquad \mu\in\M_1(\Delta_\s).
\] 
Then for any $\mu\in\M_1(\Delta_\s)$, Proposition \ref{TstarmuN} (c) provides  the existence of $\delta_\mu>0$ such that
\eq
\label{ineqepsilon}
\limsup_{N\rightarrow\infty}\frac{1}{N^2}\log\Big\{ Z_N\p_N\Big(T_*\mu^N\in\B(\mu,\delta_\mu)\Big)\Big\}\leq - I^\epsilon_{\bs\V}(\mu).
\qe
Since $\M_1(\Delta_\s)$ is compact, so is ${\rm clo}(T_*\F\big)$,  and thus there exists a finite number of measures $\mu_1,\ldots,\mu_d\in{\rm clo}(T_*\F\big)$ such that
\[
\p_N\Big(T_*\mu^N\in{\rm clo}(T_*\F)\Big)\leq \sum_{i=1}^d \p_N\Big(T_*\mu^N\in\B(\mu_i,\delta_{\mu_i})\Big).
\]
As a consequence, it follows with \eqref{ineqepsilon}
\begin{align}
\label{ineqepsilon2}
 & \limsup_{N\rightarrow\infty}\frac{1}{N^2}\log\Big\{ Z_N\p_N\Big(T_*\mu^N\in{\rm clo}(T_*\F)\Big)\Big\}\nonumber\\
\leq &\quad \max_{i=1}^d\,\limsup_{N\rightarrow\infty}\frac{1}{N^2}\log\Big\{ Z_N\p_N\Big(T_*\mu^N\in\B(\mu_i,\delta_{\mu_i})\Big)\Big\}\nonumber\\
\leq &\quad -\min_{i=1}^d  I^\epsilon_{\bs\V}(\mu_i) \quad \leq  \; -\inf_{\mu\in{\,\rm clo}(T_*\F)} I^\epsilon_{\bs\V}(\mu).
\end{align}
By letting $\epsilon\rightarrow0$ in \eqref{ineqepsilon2}, we obtain
\eq
\label{C2}
\limsup_{N\rightarrow\infty}\frac{1}{N^2}\log\Big\{ Z_N \p_N\Big(T_*\mu^N\in{\rm clo}(T_*\F)\Big)\Big\}
\leq -\inf_{\mu\in{\,\rm clo}(T_*\F)} I_{\bs\V}(\mu).
\qe
If $\nu\in{\rm clo}(T_*\F)$, then either $\nu\in T_*\F$ or $\nu(\{\infty\})>0$. Indeed, let $(T_*\eta_N)_N$ be a sequence in $T_*\F$ with limit $\nu$ satisfying $\nu(\{\infty\})=0$. Lemma \ref{homeo} yields $\eta\in\M_1(\Delta)$ such that $\nu=T_*\eta$ and moreover the convergence of  $(\eta_N)_N$  towards  $\eta$. Since $\F$ is closed, necessarily  $\nu\in T_*\F$. As a consequence, since $I_{\bs\V}(\mu)=+\infty$ as soon as $\mu(\{\infty\})>0$, we obtain from \eqref{relation}
\eq
\label{C3}
\inf_{\mu\in{\,\rm clo}(T_*\F)} I_{\bs\V}(\mu)   = \inf_{\mu\in T_*\F} I_{\bs\V}(\mu) = \inf_{\mu\in \F}I_V(\mu).
\qe
Finally,  \eqref{UB} follows from \eqref{C1}, and \eqref{C2}--\eqref{C3}.

We now prove \eqref{LB}. It is sufficient to show that for any $\mu\in\M_1(\Delta)$ and any neighborhood $\mathcal{G}\subset\M_1(\Delta)$ of $\mu$ we have 
\eq
\label{wLB}
\liminf_{N\rightarrow\infty}\frac{1}{N^2}\log \Big\{Z_N\p\Big(\mu^N\in\mathcal{G}\Big)\Big\}\geq - I_V(\mu).
\qe
For any $k$ large enough, define $\mu_k\in\M_1(\R)$ to be the normalized restriction of $\mu$ to the compact $\Delta\cap[-k,k]^2$. Then $(\mu_k)_k$ converges towards $\mu$ as $k\rightarrow\infty$ and one easily obtains from the monotone  convergence theorem
that \[
\lim_{k\rightarrow\infty}I_V(\mu_k)=I_V(\mu).
\]
As a consequence, it is enough to show \eqref{wLB} under the extra assumption that the $\mu$'s are compactly supported, so that the statement \eqref{LB} is independent of the growth assumption on $V$. Thus, one can reproduce the proof of \cite[Theorem 2.6.1]{AGZ} to show \eqref{wLB} when $\Delta=\R$, and similarly the one of \cite[Theorem 5.4.9]{HP} when $\Delta=\C$. The prove of Theorem \ref{th2} is therefore complete.
\end{proof}

\begin{remark} 
\label{alternative}
An alternative approach to the proof of Theorem \ref{th2} is as follows. Assume that one can establish a   large deviation lower bound similar to \eqref{wLB} for $T_*\mu^N$, so that it would provide together with Proposition \ref{TstarmuN} a full large deviation principle for $T_*\mu^N$ on $\M_1(\Delta_\s)$. Then one would obtain a large deviation principle for $T_*\mu^N$ on $\{\mu\in\M_1(\Delta_\s):\;\mu(\{\infty\})=0\}$, equipped with the induced topology of $\M_1(\Delta_\s)$, by "inclusion principle" \cite[Lemma 4.1.5(b)]{DZ}, and then the required large deviation principle for $\mu^N$ on $\M_1(\Delta)$ by contraction principle  along $T_*^{-1}$ \cite[Theorem 4.2.1]{DZ}, thanks to Lemma \ref{homeo}. 
\end{remark}

\section{Generalizations} 

In this section we consider some generalizations of the result and the method presented in the previous sections.
 
\label{generalizations}

\subsection{Concerning the support of the Coulomb gas}
\label{genesupport}
A natural question is to ask if Theorem \ref{th2} still holds for more general supports $\Delta$ and less regular potentials $V$, as suggested in the previous sections.

Let us emphasis that the compactification procedure presented in Section \ref{compactification}  and Proposition \ref{TstarmuN}  hold under the only assumptions that  $\Delta$ is a closed subset of $\C$ and $V:\Delta\rightarrow\R\cup\{+\infty\} $ is a lower semi-continuous  map which satisfies the growth assumption \eqref{weakgrowth}, and such that  there exists $\mu\in\M_1(\Delta)$ with $I_V(\mu)<+\infty$. As a consequence,   the proofs of Theorem \ref{th2}(a), (b) and the upper bound \eqref{UB} provided in Section \ref{proofth2} also hold under such a weakening of assumptions on $V$ and $\Delta$. A full large deviation principle would hold as soon as one can establish in this setting the lower bound \eqref{LB} for $\mu^N$, or its equivalent for $T_*\mu^N$, see Remark \ref{alternative}.

\subsection{Concerning the compactification procedure}

The main use of the compactification procedure was to avoid the use of exponential tightness to prove the large  deviation upper bound. It turns out that the proof of \eqref{UB} can be adapted without any substantial change to obtain a similar result in a more general setting that we  present now. 
 
 Let $\X$ be a locally compact, but not compact, Polish space and consider a sequence $(\mu^N)_N$ of random variables taking values in the space $\M_1(\X)$ of Borel probability measures on $\X$.   Let $(\widehat{\X},T)$ be a one-point compactification of $\X$, that is  a compact set $\widehat \X$ with an element $\infty\in\widehat \X$ such that $T : \X \rightarrow \widehat \X$ is an homeomorphism on its image $T(X)$ and $\widehat \X\setminus T(\X)=\{\infty\}$.  Define $T_*$ to be the push-forward by $T$ similarly as in \eqref{pushf}. We equip $\M_1(\widehat \X)$ with its weak topology, so that it becomes a  compact Polish space, and denotes $\B(\mu,\delta)$ the ball centered in $\mu\in\M_1(\widehat \X)$ with radius $\delta>0$.
 
\begin{proposition}
Let $(\alpha_N)_N$ and $(Z_N)_N$ be two sequences of real positive numbers with $\lim_{N\rightarrow\infty}\alpha_N=+\infty$. Assume there exists a lower semi-continuous map   $\Phi :\M_1(\widehat \X)\rightarrow\R\cup\{+\infty\}$ which satisfies the following.
 \begin{itemize}
 \item[{\rm (a)}] For all $\mu\in\M_1(\widehat \X)$, $\Phi(\mu)=+\infty$ as soon as $\mu(\{\infty\})>0$.
  \item[{\rm (b)}] For all $\mu\in\M_1(\widehat \X)$,
\[
\limsup_{\delta\rightarrow0}\limsup_{N\rightarrow\infty}\frac{1}{\alpha_N}\log\Big\{ Z_N\p_N\Big(T_*\mu^N\in\B(\mu,\delta)\Big)\Big\}\leq - \Phi(\mu).
\]
 \end{itemize}
Then for any closed set $\F\subset \M_1(\X)$,
\[
\limsup_{N\rightarrow\infty}\frac{1}{\alpha_N}\log \Big\{Z_N \p_N\Big(\mu^N\in\F\Big)\Big\}\leq - \inf_{\mu\in\F}\Phi\circ T_*(\mu).
\]
\end{proposition}
Moreover, note that $\Phi$ has compact level sets (resp. is strictly convex on the set where it is finite) if and only if $\Phi\circ T_*$ has (resp.  is). 

We mention that a similar strategy is used in \cite{HK2} where a LDP is established  for a two type particles Coulomb gas related to an additive perturbation of a Wishart random matrix model.
\\

\subsection*{Acknolwledgments}
The author is grateful to the anonymous referees for their useful suggestions and remarks, e.g. to point out Remark \ref{alternative}. He is supported by FWO-Flanders projects G.0427.09 and by the Belgian Interuniversity
Attraction Pole P06/02.

\end{document}